\documentclass[reqno,12pt]{amsart}
\usepackage{amscd,amsfonts,amssymb}
\numberwithin{equation}{section}
\newtheorem{Theorem}{Theorem}[section]
\newtheorem{Lemma}{Lemma}[section]
\newtheorem{Corollary}{Corollary}[section]

\theoremstyle{definition}
\newtheorem{Definition}{Definition}[section]
\theoremstyle{remark}
\newtheorem{Remark}{Remark}[section]
\newtheorem{Example}{Example}[section]
\newcommand{\sign}{\mathop{\rm sign}\nolimits}

\newcommand{\essinf}{\mathop{\rm ess \, inf}\limits}
\newcommand{\esssup}{\mathop{\rm ess \, sup}\limits}

\dedicatory{Dedicated to Laurent V\'eron}
\author{A.A. Kon'kov}
\address{Department of Differential Equations,
Faculty of Mechanics and Mathematics,
Mo\-s\-cow Lo\-mo\-no\-sov State University,
Vorobyovy Gory,
Moscow, 119992 Russia}
\email{konkov@mech.math.msu.su}
\author{A.E. Shishkov}
\address{
Center of Nonlinear Problems of Mathematical Physics,
RUDN University,
Miklukho-Maklaya str. 6,
Moscow, 117198 Russia;
Institute of Applied Mathematics and Mechanics of NAS of Ukraine,
Dobrovol'skogo str. 1, Slavyansk, 84116 Ukraine
}
\email{aeshkv@yahoo.com}
\thanks{The research is supported by RUDN University, Project 5-100}
\title[On removable singularities]{On removable singularities of solutions of higher order differential inequalities}
\date{}

\begin{document}

\begin{abstract}
We obtain sufficient conditions for solutions of the $m$th-order differential inequality 
$$
	\sum_{|\alpha| = m}
	\partial^\alpha
	a_\alpha (x, u)
	\ge
	f (x) g (|u|)
	\quad
	\mbox{in } B_1 \setminus \{ 0 \}
$$
to have a removable singularity at zero, 
where $a_\alpha$, $f$, and $g$ are some functions, and 
$B_1 = \{ x : |x| < 1 \}$ is a unit ball in ${\mathbb R}^n$.

Constructed examples demonstrate the exactness of these conditions.
\end{abstract}

\maketitle

\section{Introduction}

We study solutions of the differential inequality
\begin{equation}
	\sum_{|\alpha| = m}
	\partial^\alpha
	a_\alpha (x, u)
	\ge
	f (x) g (|u|)
	\quad
	\mbox{in } B_1 \setminus \{ 0 \}
	\label{1.1}
\end{equation}
of order $m \ge 1$, where $a_\alpha$ are Caratheodory functions such that
$$
	|a_\alpha (x, \zeta)| \le A |\zeta|,
	\quad
	|\alpha| = m,
$$
with some constant $A > 0$ for almost all $x \in B_1$ and for all 
$\zeta \in {\mathbb R}$.
It is assumed that $f$ is a positive measurable function 
and
$g \in C^2 ([0, \infty))$ 
satisfies the conditions
$g (\zeta) > 0$, $g' (\zeta) > 0$, and $g'' (\zeta) > 0$ for all $\zeta \in (0, \infty)$.

As is customary, by $B_r$ we denote an open ball in ${\mathbb R}^n$ 
of radius $r > 0$ centered at zero.
In so doing, by $\alpha = {(\alpha_1, \ldots, \alpha_n)}$ we mean a multi-index with
$|\alpha| = \alpha_1 + \ldots + \alpha_n$
and
$\partial^\alpha = {\partial^{|\alpha|} / (\partial_{x_1}^{\alpha_1} \ldots \partial_{x_n}^{\alpha_n}})$,
$x = {(x_1, \ldots, x_n)}$.

\begin{Definition}\label{D1.1}
A function $u$ is called a weak solution of~\eqref{1.1} if
$u  \in {L_1 (B_1 \setminus B_\varepsilon)}$
and
${f (x) g (|u|)} \in {L_1 (B_1 \setminus B_\varepsilon)}$
for all $\varepsilon \in (0, 1)$ and, moreover,
\begin{equation}
	\int_{B_1}
	\sum_{|\alpha| = m}
	(-1)^m
	a_\alpha (x, u)
	\partial^\alpha
	\varphi
	\,
	dx
	\ge
	\int_{B_1}
	f (x) g (|u|)
	\varphi
	\,
	dx
	\label{1.2}
\end{equation}
for any non-negative function $\varphi \in C_0^\infty (B_1 \setminus \{ 0 \})$.
\end{Definition}

\begin{Definition}\label{D1.2}
A weak solution of~\eqref{1.1} has a removable singularity at zero if
$u  \in {L_1 (B_1)}$, ${f (x) g (|u|)} \in {L_1 (B_1)}$,
and~\eqref{1.2} is valid for any non-negative function $\varphi \in C_0^\infty (B_1)$.
In other words, $u$ is a weak solution of the inequality
$$
	\sum_{|\alpha| = m}
	\partial^\alpha
	a_\alpha (x, u)
	\ge
	f (x) g (|u|)
	\quad
	\mbox{in } B_1.
$$
\end{Definition}

In a similar way, we can define a weak solution (and a weak solution with a removable singularity) 
of the equation  
$$
	\sum_{|\alpha| = m}
	\partial^\alpha
	a_\alpha (x, u)
	=
	f (x) g (|u|)
	\sign u
	\quad
	\mbox{in } B_1 \setminus \{ 0 \}.
$$

In the partial case of the Emden-Fowler nonlinearity $g (t) = t^\lambda$, inequality~\eqref{1.1} takes the form
\begin{equation}
	\sum_{|\alpha| = m}
	\partial^\alpha
	a_\alpha (x, u)
	\ge
	f (x) |u|^\lambda
	\quad
	\mbox{in } B_1 \setminus \{ 0 \}.
	\label{1.3}
\end{equation}

The problem of removability of an isolated singularity for solutions of differential equations and 
inequalities has traditionally attracted the attention of mathematicians.
A wide literature is devoted to this issue~[1--15].
However, most of these papers deal with second-order equations and inequalities~[1--13].
The case of higher order differential operators is studied 
mainly for nonlinearities of the Emden-Fowler type $g (t) = t^\lambda$~\cite{K, KE}.

In the present paper, we obtain sufficient conditions for weak solutions of~\eqref{1.1}
to have a removable singularity at zero.
In so doing, we are not limited to the case of the Emden-Fowler nonlinearity.

We also impose no ellipticity conditions on the coefficients $a_\alpha$ of the differential operator.
Therefore, our results can be applied to a wide class of differential inequalities.
The exactness of these results is demonstrated in Examples~\ref{E2.1}--\ref{E2.4}.

It is interesting that, in the case of the equation
\begin{equation}
	\Delta u = |u|^\lambda \sign u
	\quad
	\mbox{in } B_1 \setminus \{ 0 \},
	\label{C2.1.1}
\end{equation}
the conditions for removability of a singularity obtained in the classical paper of 
L.~V\'eron and H.~Bre\-zis~\cite{BV} coincide with the analogous conditions 
for weak solutions of the inequality
$$
	\Delta u \ge |u|^\lambda \sign u
	\quad
	\mbox{in } B_1 \setminus \{ 0 \}
$$
while the equation
\begin{equation}
	- \Delta u = |u|^\lambda \sign u
	\quad
	\mbox{in } B_1 \setminus \{ 0 \}
	\label{R2.1.1},
\end{equation}
has solutions from $C^2 (B_1 \setminus \{ 0 \})$ with a removable singularity at zero in the weak sense
which are not twice continuously differentiable functions in the whole ball $B_1$
(see Corollary~\ref{C2.1} and Remark~\ref{R2.1}).

We use the following notations. By
$$
	g^* (\xi)
	=
	\left\{
		\begin{aligned}
			&
			\int_{
				g' (0)
			}^{
				\xi
			}
			(g')^{-1} (\zeta)
			\,
			d \zeta,
			&
			&
			\xi > g' (0),
			\\
			&
			0,
			&
			&
			\xi \le g' (0),
		\end{aligned}
	\right.
$$
where $(g')^{-1}$ is the inverse function to $g'$,
we denote the Legendre transformation of the function $g (t) - g (0)$.
In accordance with the Fenchel-Young inequality we have
$$
	a b \le g (a) + g^* (b)
$$
for all real numbers $a \ge 0$ and $b \ge 0$.
In the case of $g (t) = t^\lambda / \lambda$, $\lambda > 1$, 
this inequality obviously takes the form
$$
	a b
	\le 
	\frac{
		1
	}{
		\lambda
	} 
		a^\lambda
	+ 
	\frac{
		\lambda - 1
	}{
		\lambda
	}
		b^{\lambda / (\lambda - 1)}
$$
for all real numbers $a \ge 0$ and $b \ge 0$.

Let us put
$$
	\gamma (\xi)
	=
	\frac{g^* (\xi)}{\xi}.
$$

We assume that there are a real number $\lambda \ge 1$, 
a non-negative measurable function $\rho$, and a positive non-decreasing function $h$
such that
\begin{equation}
	g (\varepsilon r^{m - n} t) 
	\ge 
	\varepsilon^\lambda 
	\rho (r) 
	h (t)
	\label{1.4}
\end{equation}
for all 
$\varepsilon \in (0, 1)$, 
$r \in (0, 1)$,
and 
$t \in (0, \infty)$.

\section{Main results}

\begin{Theorem}\label{T2.1}
Let
\begin{equation}
	\int_1^\infty
	h^{- 1 / (\lambda (m - 1) + 1)} (t)
	t^{1 / (\lambda (m - 1) + 1) - 1}
	\,
	d t
	< 
	\infty
	\label{T2.1.1}
\end{equation}
and
\begin{equation}
	\int_0^1
	r^{n - 1}
	q (r)
	\,
	dr
	=
	\infty,
	\label{T2.1.2}
\end{equation}
where
$$
	q (r)
	=
	\rho (r)
	\frac{
		\essinf_{
			B_1 \cap B_{\sigma r} \setminus B_{r / \sigma}
		} 
		f^\lambda 
	}{
		\esssup_{
			B_1 \cap B_{\sigma r} \setminus B_{r / \sigma}
		} 
		f^{\lambda - 1}
	}
$$
for some real number $\sigma > 1$.
If
\begin{equation}
	\int_{B_1}
	\gamma
	\left(
		\frac{1}{f (x)}
	\right)
	dx
	<
	\infty,
	\label{T2.1.4}
\end{equation}
then any weak solution of~\eqref{1.1} has a removable singularity at zero.
\end{Theorem}

In the case of the Emden-Fowler nonlinearity, Theorem~\ref{T2.1} implies  the following assertion.

\begin{Theorem}\label{T2.2}
Let $\lambda > 1$ and
$$
	\int_0^1
	r^{\lambda (m - n) + n - 1}
	z (r)
	\,
	dr
	=
	\infty,
$$
where
$$
	z (r)
	=
	\frac{
		\essinf_{
			B_1 \cap B_{\sigma r} \setminus B_{r / \sigma}
		} 
		f^\lambda 
	}{
		\esssup_{
			B_1 \cap B_{\sigma r} \setminus B_{r / \sigma}
		} 
		f^{\lambda - 1}
	}
$$
for some real number $\sigma > 1$.
If
$$
	\int_{B_1}
	f^{- 1 / (\lambda - 1)} (x)
	\,
	dx
	<
	\infty,
$$
then any weak solution of~\eqref{1.3} has a removable singularity at zero.
\end{Theorem}

\begin{Corollary}[H.~Bresis and L.~V\'eron~\cite{BV}]\label{C2.1}
Let $u \in C^2 (B_1  \setminus \{ 0 \})$ be a solution of~\eqref{C2.1.1},
where
\begin{equation}
	\lambda
	\ge
	\frac{n}{n - 2},
	\quad
	n \ge 3.
	\label{C2.1.2}
\end{equation}
Then $u \in C^2 (B_1)$ and, moreover,
\begin{equation}
	\Delta u = |u|^\lambda \sign u
	\quad
	\mbox{in } B_1.
	\label{C2.1.3}
\end{equation}
\end{Corollary}

\begin{proof}
By the Kato theorem~\cite{Kato}, the function $u_+ (x) = \max \{ u (x), 0 \}$ 
is a weak solution of the inequality
$$
	\Delta u_+ \ge u_+^\lambda
	\quad
	\mbox{in } B_1 \setminus \{ 0 \}.
$$
Applying Theorem~\ref{T2.2}, we obtain that $u$ is also a weak solution of the inequality
$$
	\Delta u_+ \ge u_+^\lambda
	\quad
	\mbox{in } B_1.
$$
Since the right-hand side of the last expression is non-negative, we have
\begin{equation}
	\esssup_{
		B_{1 / 2}
	}
	u_+
	\le
	\sup_{
		\partial B_{1 / 2}
	}
	u_+;
	\label{PC2.1.1}
\end{equation}
therefore, $u_+ \in L_\infty (B_{1 / 2})$.
To verify the validity of~\eqref{PC2.1.1}, it suffices to take
$$
	u_{+\varepsilon} (x)
	=
	\int_{B_1}
	\omega_\varepsilon (x - y)
	u_+ (y)
	\,
	dy,
	\quad
	\varepsilon > 0,
$$
where
$$
	\omega_\varepsilon (x) 
	= 
	\frac{1}{\varepsilon^n} 
	w
	\left(
		\frac{x}{\varepsilon}
	\right),
	\quad
	\varepsilon > 0,
$$
are Steklov-Schwartz averaging kernels 
for some non-negative function $\omega \in C_0^\infty (B_1)$ such that
$$
	\int_{B_1}
	\omega
	\,
	d x
	=
	1.
$$
It is obvious that $u_{+\varepsilon} \in C^\infty (\overline{B_{1 / 2}})$ and, moreover,
$$
	\Delta u_{+\varepsilon} \ge 0
	\quad
	\mbox{in }
	B_{1 / 2}
$$
for all $\varepsilon \in (0, 1 / 2)$. Hence, using the maximum principle, we obtain
$$
	\sup_{
		B_{1 / 2}
	}
	u_{+\varepsilon}
	\le
	\sup_{
		\partial B_{1 / 2}
	}
	u_{+\varepsilon}
$$
for all $\varepsilon \in (0, 1 / 2)$.
In the limit as $\varepsilon \to +0$, this obviously yields~\eqref{PC2.1.1}.

Analogously, one can show that $u_- (x) = \max \{ - u (x), 0 \} \in L_\infty (B_{1 / 2})$; 
therefore, $u \in L_\infty (B_{1 / 2})$. Further, putting
$$
	v (x)
	=
	- \frac{
		1
	}{
		(n - 2) |S_1|
	}
	\int_{
		B_1
	}
	\frac{
		\varphi (y) 
		|u (y)|^\lambda \sign u (y)
		\, dy
	}{
		|x - y|^{n - 2}
	},
$$
where $|S_1|$ is a $(n - 1)$-dimensional volume of the unit sphere in ${\mathbb R}^n$ and
$\varphi \in C_0^\infty (B_1)$ is some function equal to one on $B_{1 / 2}$, 
we obtain $u - v \in C^\infty (B_{1 / 2})$ since $u - v$ is a bounded harmonic function in 
$B_{1 / 2} \setminus \{ 0 \}$.
The condition $u \in C^2 (B_1 \setminus \{ 0 \}) \cap L_\infty (B_{1 / 2})$ 
implies that $v \in C^1 (B_{1 / 2})$. 
Therefore, $u$ belongs to $C^2 (B_1 \setminus \{ 0 \}) \cap C^1 (B_{1 / 2})$.
This, in turn, implies that $v \in C^2 (B_{1 / 2})$ and, accordingly,  $u \in C^2 (B_1)$. 
Consequently, $u$ satisfies equation~\eqref{C2.1.3} in the classical sense.
\end{proof}

\begin{Remark}\label{R2.1}
Condition~\eqref{C2.1.2} also guarantees the removability 
of singularity at zero for weak non-negative solutions of~\eqref{R2.1.1}
since, in Theorem~\ref{T2.2}, it does not matter what sign the Laplace operator faces.
However, unlike~\eqref{C2.1.1}, we can not argue that these solutions belongs to $C^2 (B_1)$
even if~\eqref{R2.1.1} is understood in the classical sense.
In fact, if
$$
	\lambda 
	>
	\frac{n}{n - 2},	
$$
then~\eqref{R2.1.1} has a solution of the form
$$
	u (x) = c |x|^{- 2 / (\lambda - 1)},
$$
where $c > 0$ is some constant~\cite[Theorem~1.3]{GS}. 
The function $u$ is a weak solution of the equation
\begin{equation}
	- \Delta u = |u|^\lambda \sign u
	\quad
	\mbox{in } B_1,
	\label{R2.1.2}
\end{equation}
but it does not satisfy~\eqref{R2.1.2} in the classical sense.
\end{Remark}

Theorems~\ref{T2.1} and~\ref{T2.2} are proved in Section~\ref{proofOfTheorems}.
Now, let us demonstrate their exactness.

\begin{Example}\label{E2.1}
Consider the inequality
\begin{equation}
	\Delta^{m / 2} u \ge c |x|^s |u|^\lambda
	\quad
	\mbox{in } B_1 \setminus \{ 0 \},
	\quad
	c = const > 0,
	\label{E2.1.1}
\end{equation}
where $\lambda$ and $s$ are real numbers and $m$ is a positive even integer.
By Theorem~\ref{T2.2}, if
\begin{equation}
	\lambda > 1
	\quad
	\mbox{and}
	\quad
	s \le \lambda (n - m) - n,
	\label{E2.1.2}
\end{equation}
then any weak solution of~\eqref{E2.1.1} has a removable singularity at zero.
For $m = 2$, condition~\eqref{E2.1.2} coincides with the analogous condition
given in~\cite[Example~6.1.1]{meCMFI}.
In turn, if $m = 2$, $s = 0$, and $n \ge 3$, then~\eqref{E2.1.2} coincides with~\eqref{C2.1.2}. 

Let us examine the critical exponent $s = \lambda (n - m) - n$ 
in the right-hand side of~\eqref{E2.1.1}. 
Namely, assume that $u$ is a weak solution of the inequality
\begin{equation}
	\Delta^{m / 2} u 
	\ge 
	c |x|^{\lambda (n - m) - n} 
	\log^\nu \frac{1}{|x|}
	\,
	|u|^\lambda
	\quad
	\mbox{in } B_1 \setminus \{ 0 \},
	\quad
	c = const > 0,
	\label{E2.1.3}
\end{equation}
where $\lambda$ and $\nu$ are real numbers and $m$ is a positive even integer.
By Theorem~\ref{T2.2}, if
\begin{equation}
	\lambda > 1
	\quad
	\mbox{and}
	\quad
	\nu \ge - 1,
	\label{E2.1.4}
\end{equation}
then $u$ has a removable singularity at zero.
For $m = 2$, condition~\eqref{E2.1.4} coincides with the analogous condition obtained 
in~\cite[Example~6.1.2]{meCMFI}

It can be seen that, in the case of $\lambda \le 1$, 
for any $c$, $s$, and $\nu$ there exist real numbers $k > 0$ and $l > 0$ such that
\begin{equation}
	u (x)
	=
	e^{
		k / |x|^l
	}
	\label{E2.1.5}
\end{equation}
is a weak solution of both~\eqref{E2.1.1} and~\eqref{E2.1.3} with an unremovable singularity at zero.
Therefore, the first inequality in~\eqref{E2.1.2} and~\eqref{E2.1.4} is exact.

Assume now that $\lambda > 1$ and $\nu < -1$. Let us put
$$
	w_0 (r)
	=
	r^{-n}
	\log^{- (\nu + \lambda) / (\lambda - 1)} \frac{1}{r}
$$
and
$$
	w_i (r)
	=
	\frac{1}{n - 2}
	\int_r^1
	\left(
		\left(
			\frac{\zeta}{r}
		\right)^{n - 2}
		-
		1
	\right)
	\zeta
	w_{i - 1} (\zeta)
	\,
	d\zeta,
	\quad
	i = 1, \ldots, m / 2.
$$
It is obvious that
$$
	\Delta w_i (|x|) = w_{i - 1} (|x|)
	\quad
	\mbox{in } B_1 \setminus \{ 0 \},
	\quad
	i = 1, \ldots, m / 2.
$$
In the case of $m < n$, for any $1 \le i \le m / 2$ we also have
$$
	w_i (r) \sim r^{2 i - n} \log^{- (1 + \nu) / (\lambda - 1)} \frac{1}{r}
	\quad
	\mbox{as } r \to +0,
$$
or in other words,
$$
	c_1 r^{2 i - n} \log^{- (1 + \nu) / (\lambda - 1)} \frac{1}{r}
	\le
	w_i (r) 
	\le 
	c_2 r^{2 i - n} \log^{- (1 + \nu) / (\lambda - 1)} \frac{1}{r}
$$
with some constants $c_1 > 0$ and $c_2 > 0$ for all $r > 0$ in a neighborhood of zero.
Hence, for any $c > 0$ there are real numbers $\varepsilon > 0$ and $\delta > 0$ such that the function
\begin{equation}
	u (x)
	=
	\varepsilon
	w_{m / 2} 
	(
		\delta |x|
	)
	\label{E2.1.6}
\end{equation}
is a weak solution of~\eqref{E2.1.3} with an unremovable singularity at zero.
Thus, the second inequality in~\eqref{E2.1.4} is exact for all $m < n$.
Since solutions of~\eqref{E2.1.3} are also solutions of~\eqref{E2.1.1} for any $s > \lambda (n - m) - n$,
we have simultaneously showed the exactness of the second inequality in~\eqref{E2.1.2}.
\end{Example}

\begin{Example}\label{E2.2}
We examine the critical exponent $\lambda = 1$ in the right-hand side of~\eqref{E2.1.1}.
Consider the inequality
\begin{equation}
	\Delta^{m / 2} u 
	\ge 
	c |x|^s
	|u| \log^\nu (e + |u|)
	\quad
	\mbox{in } B_1 \setminus \{ 0 \},
	\quad 
	c = const > 0,
	\label{E2.2.1}
\end{equation}
where $\nu$ and $s$ are real numbers and $m \le n$ is a positive even integer.

By Theorem~\ref{T2.1}, if
\begin{equation}
	\nu > m
	\quad
	\mbox{and}
	\quad
	s \le - m,
	\label{E2.2.2}
\end{equation}
then any weak solution of~\eqref{E2.2.1} has a removable singularity at zero.
In fact, for any $\delta \in (0, 1)$ there exists $\kappa \in (0, \infty)$ such that
\begin{equation}
	\log (e + \varepsilon \tau)
	\ge
	\kappa
	\varepsilon^\delta
	\log (e + \tau)
	\label{E2.2.3}
\end{equation}
for all $\varepsilon \in (0, 1)$ and $\tau \in (0, \infty)$.
To establish the validity of the last inequality, we assume the converse. 
Then there are a sequences or real numbers $\varepsilon_i \in (0, 1)$ and $\tau_i \in (0, \infty)$ such that
\begin{equation}
	\log (e + \varepsilon_i \tau_i)
	<
	\frac{\varepsilon_i^\delta}{i}	
	\log (e + \tau_i),
	\quad
	i = 1,2,\ldots.
	\label{E2.2.4}
\end{equation}
It is clear that $\varepsilon_i \tau_i \to \infty$ as $i \to \infty$; 
otherwise there are subsequences $\varepsilon_{i_j}$ and $\tau_{i_j}$ such that
$\tau_{i_j} \le \beta / \varepsilon_{i_j}$ with some constant $\beta > 0$ for all $j = 1, 2, \ldots$.
Hence, taking into account~\eqref{E2.2.4}, we arrive at a contradiction.
In particular, one can assert that $\tau_i \to \infty$ as $i \to \infty$; 
therefore,~\eqref{E2.2.4} implies the inequality
$$
	\log \varepsilon_i + \log \tau_i
	<
	\frac{2 \varepsilon_i^\delta}{i}
	\log \tau_i
$$
for all sufficiently large $i$, whence it follows that
$$
	\log \tau_i
	<
	\frac{
		1
	}{
		1 - 2 \varepsilon_i^\delta / i
	}
	\log \frac{1}{\varepsilon_i}
$$
for all sufficiently large $i$.
Thus, 
$$
	\log (\varepsilon_i \tau_i)
	<
	\frac{
		2 \varepsilon_i^\delta / i
	}{
		1 - 2 \varepsilon_i^\delta / i
	}
	\log \frac{1}{\varepsilon_i}
	\to
	0
	\quad
	\mbox{as } i \to \infty
$$
and we again arrive at a condition.

From~\eqref{E2.2.3}, it follows that
$$
	\log (e + \varepsilon r^{m - n} t)
	\ge
	\kappa
	\varepsilon^\delta
	\log (e + r^{m - n} t)
	\ge
	\kappa
	\varepsilon^\delta
	\log (e + t)
$$
for all 
$\varepsilon \in (0, 1)$, 
$r \in (0, 1)$,
and 
$t \in (0, \infty)$.
Therefore, taking
$$
	g (\zeta) = \zeta \log^\nu (e + \zeta),
$$
we obtain~\eqref{1.4} with
$$
	\lambda = 1 + \delta \nu,
	\quad
	\rho (r) = \kappa^\nu r^{m - n},
	\quad
	\mbox{and}
	\quad
	h (t) = t \log^\nu (e + t).
$$

To complete our arguments, it is sufficient to note that~\eqref{E2.2.2} guarantees 
the validity of conditions~\eqref{T2.1.1}, \eqref{T2.1.2}, and \eqref{T2.1.4}
if $\delta$ is small enough.

We also note that both inequalities in~\eqref{E2.2.2} are exact.
In fact, if $\nu \le m$, then for any $c$ and $s$ there exist real numbers $k > 0$ and $l > 0$ 
such that
$$
	u (x)
	=
	e^{
		e^{
			k / |x|^l
		}
	}
$$
is a weak solution of~\eqref{E2.2.1} with an unremovable singularity at zero.
In turn, if $\nu > m$ and $s > - m$, then for any $c > 0$ there exists a real number $k > 0$ 
such that the function
$$
	u (x)
	=
	e^{
		k
		|x|^{
			(s + m) / (m - \nu)
		}
	}
$$
is a weak solution of~\eqref{E2.2.1} with an unremovable singularity at zero.
\end{Example}

\begin{Example}\label{E2.3}
Consider the inequality
\begin{equation}
	\Delta^{m / 2} u 
	\ge 
	c |x|^{\lambda (n - m) - n}
	|u|^\lambda \log^\nu (e + |u|)
	\quad
	\mbox{in } B_1 \setminus \{ 0 \},
	\quad 
	c = const > 0,
	\label{E2.3.1}
\end{equation}
where $\lambda$ and $\nu$ are real numbers and $m < n$ is a positive even integer.

We are interested in the case of $\lambda > 1$. By Theorem~\ref{T2.1}, if
\begin{equation}
	\nu \ge - 1,
	\label{E2.3.3}
\end{equation}
then any weak solution of~\eqref{E2.3.1} has a removable singularity at zero.
Indeed, let condition~\eqref{E2.3.3} be valid.
Without loss of generality, it can be assumed that $\nu < 0$; 
otherwise we replace $\nu$ by $-1$. 
After this replacement, inequality~\eqref{E2.3.1} obviously remains valid.

We have
$$
	\log (e + a b)
	\le
	\log (e + a^2)
	\log (e + b^2)
$$
for all real numbers $a > 0$ and $b > 0$. This allows us to assert that
$$
	\log (e + \varepsilon r^{m - n} t)
	\le
	\log (e + r^{m - n} t)
	\le
	\log (e + r^{2 (m - n)})
	\log (e + t^2)
$$
for all 
$\varepsilon \in (0, 1)$, 
$r \in (0, 1)$,
and 
$t \in (0, \infty)$.
Therefore, taking
$$
	g (\zeta)
	=
	\zeta^\lambda 
	\log^\nu (e + \zeta),
$$
we obtain~\eqref{1.4} with
$$
	\rho (r) 
	= 
	r^{\lambda (m - n)}
	\log^\nu (e + r^{2 (m - n)})
	\quad
	\mbox{and}
	\quad
	h (t)
	=
	t^\lambda
	\log^\nu (e + t^2).
$$
In so doing, it can be verified that~\eqref{T2.1.1}, \eqref{T2.1.2}, and \eqref{T2.1.4} hold.

Inequality~\eqref{E2.3.3} is exact for all $m < n$. 
In fact, if $\nu < -1$, then for any $\lambda > 1$ and $c > 0$ 
there exist $\varepsilon > 0$ and $\delta > 0$ such that
the function $u$ defined by~\eqref{E2.1.6} is a weak solution of~\eqref{E2.3.1} 
with an unremovable singularity at zero.
\end{Example}

\begin{Example}\label{E2.4}
Consider the first-order differential inequality
\begin{equation}
	- \sum_{i=1}^n
	\frac{\partial}{\partial x_i}
	\left(
		\frac{x_i}{|x|}
		u
	\right)
	\ge
	c
	|x|^s
	|u|^\lambda,
	\quad
	\mbox{in } B_1 \setminus \{ 0 \},
	\quad 
	c = const > 0,
	\label{E2.4.1}
\end{equation}
where $\lambda$ and $s$ are real numbers.
By Theorem~\ref{T2.2}, if
\begin{equation}
	\lambda > 1
	\quad
	\mbox{and}
	\quad
	s \le \lambda (n - 1) - n,
	\label{E2.4.2}
\end{equation}
then any weak solution of~\eqref{E2.4.1} has a removable singularity at zero.
In can easily be seen that condition~\eqref{E2.4.2} is exact. 
Indeed, if $\lambda \le 1$, then for any real numbers $c$ and $s$ 
there exist a weak solution of~\eqref{E2.4.1} with an unremovable singularity at zero.
As such a solution, we can take the function $u$ defined by~\eqref{E2.1.5},
where $k > 0$ and $l > 0$ are sufficiently large real numbers.
At the same time, if $\lambda > 1$ and $s > \lambda (n - 1) - n$, 
then for any $c > 0$ there exists $\varepsilon > 0$ such that
$$
	u (x) 
	= 
	\varepsilon 
	|x|^{
		- (s + 1) / (\lambda - 1)
	}
$$
is a weak solution of~\eqref{E2.4.1} with an unremovable singularity at zero.
\end{Example}

\section{Proof of Theorems~\ref{T2.1} and~\ref{T2.2}}\label{proofOfTheorems}

In this section, we assume that $u$ is a weak solution of inequality~\eqref{1.1}. 
Let us denote $\tau = \sigma^{1/2}$ and
$$
	E (r)
	=
	\int_{
		B_{1 / 2} \setminus B_r
	}
	f (x) g (|u|)
	\,
	dx,
	\quad
	0 < r < 1 / 2.
$$
If $E (r) = 0$ for all $r \in (0, 1 / 2)$, then $u = 0$ almost everywhere in $B_{1/2}$. 
In this case, $u$ obviously has a removable singularity at zero. 
Therefore, we can assume without loss of generality that $E (r_0) > 0$ for some $r_0 \in (0, 1 / 2)$.
By definition, put 
$$
	r_i
	=
	\inf 
	\{
		r \in (r_{i - 1} / \tau, r_{i - 1})
		:
		E (r)
		\le 
		2 E (r_{i - 1})
	\},
	\quad
	i = 1,2,\ldots.
$$
It does not present any particular problem to verify that $r_i \to 0$ as $i \to \infty$.

In all estimates given below, by $C$ and $k$ we mean omnifarious positive constants independent of $i$ and $j$.

\begin{Lemma}\label{L3.1}
For any integer $i \ge 0$ the estimate
\begin{equation}
	\int_{
		B_1 \setminus B_{1 / 2}
	}
	|u|
	\,
	dx
	+
	\frac{
		1
	}{
		(r_i - r_{i+1})^m
	}
	\int_{
		B_{r_i} \setminus B_{r_{i+1}}
	}
	|u|
	\,
	dx
	\ge
	C
	\int_{
		B_{1 / 2} \setminus B_{r_i}
	}
	f (x)
	g (|u|)
	\,
	dx
	\label{L3.1.1}
\end{equation}
is valid.
\end{Lemma}

\begin{proof}
It is sufficient to take
$$
	\varphi (x)
	=
	\psi
	\left(
		\frac{
			|x| - r_{i+1}
		}{
			r_i - r_{i+1}
		}
	\right)
	\psi
	(
		2 (1 - |x|)
	)
$$
as a test function in~\eqref{1.2}, where $\psi \in C^\infty ({\mathbb R})$ is 
a non-negative function such that
$$
	\left.
		\psi
	\right|_{
		(- \infty, 0]
	}
	=
	0
	\quad
	\mbox{and}
	\quad
	\left.
		\psi
	\right|_{
		[1, \infty)
	}
	=
	1.
$$
\end{proof}

\begin{Lemma}\label{L3.2}
Let 
$$
	\frac{
		1
	}{
		(r_{i_j} - r_{i_j + 1})^m
	}
	\int_{
		B_{r_{i_j}} \setminus B_{r_{i_j + 1}}
	}
	|u|
	\,
	dx
	\to
	0
	\quad
	\mbox{as } j \to \infty
$$
for some sequence of positive integers $\{ i_j \}_{j = 1}^\infty$.
Then $u$ has a removable singularity at zero.
\end{Lemma}

\begin{proof}
Applying the Fenchel-Young inequality, we obtain
$$
	\int_{
		B_{1 / 2}
	}
	|u|
	\,
	dx
	\le
	\int_{
		B_{1 / 2}
	}
	f (x) 
	g (|u|)
	\,
	dx
	+
	\int_{
		B_{1 / 2}
	}
	f (x) 
	g^* 
	\left(
		\frac{1}{f (x)}
	\right)
	\,
	dx.
$$
In so doing,
$$
	\int_{
		B_{1 / 2}
	}
	f (x)
	g (|u|)
	\,
	dx
	<
	\infty
$$
according to Lemma~\ref{L3.1} and
$$
	\int_{
		B_1
	}
	f (x) 
	g^* 
	\left(
		\frac{1}{f (x)}
	\right)
	\,
	dx
	=
	\int_{B_1}
	\gamma
	\left(
		\frac{1}{f (x)}
	\right)
	dx
	<
	\infty
$$
according to condition~\eqref{T2.1.4}.
Therefore, one can assert that
$$
	\int_{
		B_{1 / 2}
	}
	|u|
	\,
	dx
	<
	\infty.
$$

Let $\psi$ be the function defined in the proof of Lemma~\ref{L3.1} and 
$\varphi \in C_0^\infty (B_1)$ an arbitrary non-negative function.
Put
$$
	\varphi_j (x)
	=
	\psi_j (x)
	\varphi (x),
$$
where 
$$
	\psi_j (x)
	=
	\psi
	\left(
		\frac{
			|x| - r_{i_j + 1}
		}{
			r_{i_j} - r_{i_j + 1}
		}
	\right),
	\quad
	j = 1, 2, \ldots.
$$
We obviously have
\begin{equation}
	\int_{B_1}
	\sum_{|\alpha| = m}
	(-1)^m
	a_\alpha (x, u)
	\partial^\alpha
	\varphi_j
	\,
	dx
	\ge
	\int_{B_1}
	f (x) g (|u|)
	\varphi_j
	\,
	dx,
	\quad
	j = 1, 2, \ldots.
	\label{PL3.2.1}
\end{equation}
From Lebesgue's bounded convergence theorem, it follows that
$$
	\int_{B_1}
	f (x) g (|u|)
	\varphi_j
	\,
	dx
	\to
	\int_{B_1}
	f (x) g (|u|)
	\varphi
	\,
	dx
	\quad
	\mbox{as } j \to \infty
$$
and
$$
	\int_{B_1}
	\sum_{|\alpha| = m}
	(-1)^m
	a_\alpha (x, u)
	\psi_j
	\partial^\alpha
	\varphi
	\,
	dx
	\to
	\int_{B_1}
	\sum_{|\alpha| = m}
	(-1)^m
	a_\alpha (x, u)
	\partial^\alpha
	\varphi
	\,
	dx
	\quad
	\mbox{as } j \to \infty.
$$
Since
\begin{align*}
	&
	\left|
		\int_{B_1}
		\sum_{|\alpha| = m}
		(-1)^m
		a_\alpha (x, u)
		\partial^\alpha
		\varphi_j
		\,
		dx
		-
		\int_{B_1}
		\sum_{|\alpha| = m}
		(-1)^m
		a_\alpha (x, u)
		\psi_j
		\partial^\alpha
		\varphi
		\,
		dx
	\right|
	\\
	&
	\quad
	{}
	\le
	\frac{
		C
	}{
		(r_{i_j} - r_{i_j + 1})^m
	}
	\int_{
		B_{r_{i_j}} \setminus B_{r_{i_j + 1}}
	}
	|u|
	\,
	dx
	\to
	0
	\quad
	\mbox{as } j \to \infty,
\end{align*}
we also obtain
$$
	\int_{B_1}
	\sum_{|\alpha| = m}
	(-1)^m
	a_\alpha (x, u)
	\partial^\alpha
	\varphi_j
	\,
	dx
	\to
	\int_{B_1}
	\sum_{|\alpha| = m}
	(-1)^m
	a_\alpha (x, u)
	\partial^\alpha
	\varphi
	\,
	dx
	\quad
	\mbox{as } j \to \infty.
$$
Thus,~\eqref{PL3.2.1} implies~\eqref{1.2}. The proof is completed.
\end{proof}

\begin{Lemma}\label{L3.3}
Let
\begin{equation}
	\liminf_{i \to \infty}
	\frac{
		1
	}{
		(r_i - r_{i+1})^m
	}
	\int_{
		B_{r_i} \setminus B_{r_{i+1}}
	}
	|u|
	\,
	dx
	>
	0,
	\label{L3.3.1}
\end{equation}
then
\begin{align}
	E (r_{i+1}) - E (r_i)
	\ge
	{}
	&
	C
	(r_i - r_{i+1})^{
		\lambda (m - 1) + 1
	} 
	r_i^{
		- \lambda (m - 1) + n - 1
	}
	\sup_{
		(r_{i+1} / \tau, r_i \tau)
	}
	\rho 
	\nonumber
	\\
	&
	{}
	\times
	\frac{
		\essinf_{
			B_{r_i} \setminus B_{r_{i+1}}
		} 
		f^\lambda 
	}{
		\esssup_{
			B_{r_i} \setminus B_{r_{i+1}}
		} 
		f^{\lambda - 1}
	}
	h (k E (r_{i+1}))
	\label{L3.3.2}
\end{align}
for all sufficiently large $i$.
\end{Lemma}

\begin{proof}
In view of~\eqref{L3.3.1}, Lemma~\ref{L3.1} allows us to assert that
\begin{equation}
	\int_{
		B_{r_i} \setminus B_{r_{i+1}}
	}
	|u|
	\,
	dx
	\ge
	C
	(r_i - r_{i+1})^m
	E (r_i)
	\label{PL3.3.0}
\end{equation}
for all sufficiently large $i$.
In so doing, we admit that the constant $C > 0$ in the last expression can depend on the limit 
in the left-hand side of~\eqref{L3.3.1} and on the first summand in the left-hand side of~\eqref{L3.1.1}.
For us, it is only important that this constant does not depend on $i$.

Since
$$
	\int_{
		B_{r_i} \setminus B_{r_{i+1}}
	}
	f (x)
	|u|
	\,
	dx
	\ge
	\essinf_{
		B_{r_i} \setminus B_{r_{i+1}}
	}
	f
	\int_{
		B_{r_i} \setminus B_{r_{i+1}}
	}
	|u|
	\,
	dx,
$$
inequality~\eqref{PL3.3.0} implies the estimate
$$
	\int_{
		B_{r_i} \setminus B_{r_{i+1}}
	}
	f (x)
	|u|
	\,
	dx
	\ge
	C
	(r_i - r_{i+1})^m
	\essinf_{
		B_{r_i} \setminus B_{r_{i+1}}
	}
	f
	E (r_i),
$$
whence it follows that
\begin{equation}
	g
	\left(
		\frac{
			\int_{
				B_{r_i} \setminus B_{r_{i+1}}
			}
			f (x)
			|u|
			\,
			dx
		}{
			\int_{
				B_{r_i} \setminus B_{r_{i+1}}
			}
			f (x)
			\,
			dx
		}
	\right)
	\ge
	g
	\left(
		\frac{
			C
			(r_i - r_{i+1})^m
			\essinf_{
				B_{r_i} \setminus B_{r_{i+1}}
			}
			f
		}{
			\int_{
				B_{r_i} \setminus B_{r_{i+1}}
			}
			f (x)
			\,
			dx
		}
		E (r_i)
	\right)
	\label{PL3.3.1}
\end{equation}
for all sufficiently large $i$.

Let $i$ be a positive integer for which~\eqref{PL3.3.1} is valid.
We take $r \in (r_{i+1} / \tau, r_i \tau)$ satisfying the condition
$$
	\rho (r) 
	\ge
	\frac{1}{2}
	\sup_{
		(r_{i+1} / \tau, r_i \tau)
	}
	\rho.
$$
Since $g$ is a convex function, we have
\begin{equation}
	\frac{
		\int_{
			B_{r_i} \setminus B_{r_{i+1}}
		}
		f (x)
		g (|u|)
		\,
		dx
	}{
		\int_{
			B_{r_i} \setminus B_{r_{i+1}}
		}
		f (x)
		\,
		dx
	}
	\ge
	g
	\left(
		\frac{
			\int_{
				B_{r_i} \setminus B_{r_{i+1}}
			}
			f (x)
			|u|
			\,
			dx
		}{
			\int_{
				B_{r_i} \setminus B_{r_{i+1}}
			}
			f (x)
			\,
			dx
		}
	\right).
	\label{PL3.3.2}
\end{equation}
At the same time, it can be seen that
$$
	\int_{
		B_{r_i} \setminus B_{r_{i+1}}
	}
	f (x)
	\,
	dx
	\ge
	(r_i^n - r_{i+1}^n)
	|B_1|
	\essinf_{
		B_{r_i} \setminus B_{r_{i+1}}
	}
	f
	>
	(r_i - r_{i+1}) r_{i+1}^{n-1}
	|B_1|
	\essinf_{
		B_{r_i} \setminus B_{r_{i+1}}
	}
	f,
$$
where $|B_1|$ is the volume of the unit ball in ${\mathbb R}^n$; therefore, 
taking into account the inequalities $r_i - r_{i+1} < \sigma r$ and $r < \sigma r_{i+1}$, 
we obtain
$$
	\frac{
		(r_i - r_{i+1})^m
		\essinf_{
			B_{r_i} \setminus B_{r_{i+1}}
		}
		f
	}{
		\int_{
			B_{r_i} \setminus B_{r_{i+1}}
		}
		f (x)
		\,
		dx
	}
	<
	\frac{
		(r_i - r_{i+1})^{m - 1}
	}{
		r_{i+1}^{n-1}
		|B_1|
	}
	<
	\frac{
		\sigma^{m + n - 2}
		r^{m - n}
	}{
		|B_1|
	}.
$$
Thus, using condition~\eqref{1.4} with
$$
	\varepsilon
	=
	\frac{
		(r_i - r_{i+1})^m
		|B_1|
		\essinf_{
			B_{r_i} \setminus B_{r_{i+1}}
		}
		f
	}{
		\sigma^{m + n - 2}
		r^{m - n}
		\int_{
			B_{r_i} \setminus B_{r_{i+1}}
		}
		f (x)
		\,
		dx
	},
$$
we can estimate the right-hand side of~\eqref{PL3.3.1} as follows:
$$
	g
	\left(
		\frac{
			C
			(r_i - r_{i+1})^m
			\essinf_{
				B_{r_i} \setminus B_{r_{i+1}}
			}
			f
		}{
			\int_{
				B_{r_i} \setminus B_{r_{i+1}}
			}
			f (x)
			\,
			dx
		}
		E (r_i)
	\right)
	\ge
	\varepsilon^\lambda
	\rho (r)
	h 
	\left(
		\frac{
			C
			\sigma^{m + n - 2}
		}{
			|B_1|
		}
		E (r_i)
	\right).
$$
Combining this with~\eqref{PL3.3.1} and~\eqref{PL3.3.2}, one can conclude that
\begin{align*}
	\int_{
		B_{r_i} \setminus B_{r_{i+1}}
	}
	f (x)
	g (|u|)
	\,
	dx
	\ge
	{}
	&
	\frac{
		C
		(r_i - r_{i+1})^{\lambda m}
		r^{\lambda (n - m)}
		\rho (r)
		\essinf_{
			B_{r_i} \setminus B_{r_{i+1}}
		}
		f^\lambda 
	}{
		\left(
			\int_{
				B_{r_i} \setminus B_{r_{i+1}}
			}
			f (x)
			\,
			dx
		\right)^{\lambda - 1}
	}
	\\
	&
	{}
	\times
	h 
	\left(
		\frac{
			C
			\sigma^{m + n - 2}
		}{
			|B_1|
		}
		E (r_i)
	\right),
\end{align*}
whence due to the inequalities $r_i / \sigma < r < r_i \tau$, $2 E (r_i) \ge E (r_{i+1})$, and
$$
	\int_{
		B_{r_i} \setminus B_{r_{i+1}}
	}
	f (x)
	\,
	dx
	\le
	(r_i - r_{i+1}) r_i^{n-1}
	|B_1|
	\esssup_{
		B_{r_i} \setminus B_{r_{i+1}}
	}
	f
$$
we immediately arrive at~\eqref{L3.3.2}.
\end{proof}

From now on, we denote
$$
	\zeta_j 
	= 
	\tau^{- j} r_0,
	\quad
	j = 1, 2, \ldots.
$$

\begin{Lemma}\label{L3.4}
Let~\eqref{L3.3.1} hold, then there exists a positive integer $j_0$ such that for all $j > j_0$ 
at least one of the following two estimates is valid:
\begin{equation}
	\int_{
		E (\zeta_{j - 1})
	}^{
		E (\zeta_{j + 2})
	}
	\frac{
		d \zeta
	}{
		h (k \zeta)
	}
	\ge
	C
	\zeta_j^n
	\sup_{
		(\zeta_{j + 1}, \zeta_j)
	}
	\rho 
	\frac{
		\essinf_{
			B_{\zeta_{j - 1}} \setminus B_{\zeta_{j + 2}}
		} 
		f^\lambda 
	}{
		\esssup_{
			B_{\zeta_{j - 1}} \setminus B_{\zeta_{j + 2}}
		} 
		f^{\lambda - 1}
	},
	\label{L3.4.1}
\end{equation}
\begin{align}
	&
	\int_{
		E (\zeta_{j - 1})
	}^{
		E (\zeta_{j + 2})
	}
	h^{- 1 / (\lambda (m - 1) + 1)} (k \zeta)
	\zeta^{1 / (\lambda (m - 1) + 1) - 1}
	\,
	d \zeta
	\nonumber
	\\
	&
	\quad
	{}
	\ge
	C
	\left(
		\zeta_j^n
		\sup_{
			(\zeta_{j + 1}, \zeta_j)
		}
		\rho 
		\frac{
			\essinf_{
				B_{\zeta_{j - 1}} \setminus B_{\zeta_{j + 2}}
			} 
			f^\lambda 
		}{
			\esssup_{
				B_{\zeta_{j - 1}} \setminus B_{\zeta_{j + 2}}
			} 
			f^{\lambda - 1}
		}
	\right)^{ 
		1 / (\lambda (m - 1) + 1)
	}.
	\label{L3.4.2}
\end{align}
\end{Lemma}

\begin{proof}
We take $j_0$ such that~\eqref{L3.3.2} is valid
for all $i$ satisfying the condition $r_i \le \zeta_{j_0}$.
In view of Lemma~\ref{L3.3}, such a $j_0$ obviously exists.
Assume further that $j > j_0$ is some integer. 
By $\Xi$ we denote the set of non-negative integers $i$ for which
${(\zeta_{j + 1}, \zeta_j)} \cap {(r_{i+1}, r_i)} \ne \emptyset$.

At first, let there be $i \in \Xi$ such that $r_i = \tau r_{i+1}$.
According to~\eqref{L3.3.2}, we have
$$
	\frac{
		E (r_{i+1}) - E (r_i)
	}{
		h (k E (r_{i+1}))
	}
	\ge
	C
	r_i^n
	\sup_{
		(r_{i+1} / \tau, r_i \tau)
	}
	\rho 
	\frac{
		\essinf_{
			B_{r_i} \setminus B_{r_{i+1}}
		} 
		f^\lambda 
	}{
		\esssup_{
			B_{r_i} \setminus B_{r_{i+1}}
		} 
		f^{\lambda - 1}
	}.
$$
Thus, to verify the validity of~\eqref{L3.4.1}, it suffices to use the inequalities
$$
	\int_{
		E (\zeta_{j - 1})
	}^{
		E (\zeta_{j + 2})
	}
	\frac{
		d \zeta
	}{
		h (k \zeta)
	}
	\ge
	\int_{
		E (r_i)
	}^{
		E (r_{i + 1})
	}
	\frac{
		d \zeta
	}{
		h (k \zeta)
	}
	\ge
	\frac{
		E (r_{i + 1}) - E (r_i)
	}{
		h (k E (r_{i + 1}))
	},
$$
\begin{equation}
	\sup_{
		(r_{i + 1} / \tau, r_i \tau)
	}
	\rho 
	\ge
	\sup_{
		(\zeta_{j + 1}, \zeta_j)
	}
	\rho,
	\label{PL3.4.1}
\end{equation}
and
\begin{equation}
	\frac{
		\essinf_{
			B_{r_i} \setminus B_{r_{i + 1}}
		} 
		f^\lambda 
	}{
		\esssup_{
			B_{r_i} \setminus B_{r_{i + 1}}
		} 
		f^{\lambda - 1}
	}
	\ge
	\frac{
		\essinf_{
			B_{\zeta_{j - 1}} \setminus B_{\zeta_{j + 2}}
		} 
		f^\lambda 
	}{
		\esssup_{
			B_{\zeta_{j - 1}} \setminus B_{\zeta_{j + 2}}
		} 
		f^{\lambda - 1}
	}
	\label{PL3.4.2}
\end{equation}
arising from the inclusions
$(r_{i + 1}, r_i) \subset (\zeta_{j + 2}, \zeta_{j - 1})$
and
$(\zeta_{j + 1}, \zeta_j) \subset {(r_{i + 1} / \tau, r_i \tau)}$.

Now, let $r_i < \tau r_{i+1}$ for all $i \in \Xi$. 
In this case, we have $E (r_{i+1}) = 2 E (r_i)$ for any $i \in \Xi$. 
Hence,~\eqref{L3.3.2} implies the estimate
\begin{align*}
	&
	\left(
		\frac{
			E (r_{i+1})
		}{
			h (k E (r_{i+1}))
		}
	\right)^{1 / (\lambda (m - 1) + 1)}
	\ge
	C
	(r_i - r_{i+1})
	\\
	&
	\qquad
	{}
	\times
	\left(
	r_i^{
		- \lambda (m - 1) + n - 1
	}
	\sup_{
		(r_{i+1} / \tau, r_i \tau)
	}
	\rho 
	\frac{
		\essinf_{
			B_{r_i} \setminus B_{r_{i+1}}
		} 
		f^\lambda 
	}{
		\esssup_{
			B_{r_i} \setminus B_{r_{i+1}}
		} 
		f^{\lambda - 1}
	}
	\right)^{ 
		1 / (\lambda (m - 1) + 1)
	}
\end{align*}
for all $i \in \Xi$ from which, taking into account~\eqref{PL3.4.1} and \eqref{PL3.4.2} and the inequalities
$$
	\int_{
		E (r_i)
	}^{
		E (r_{i + 1})
	}
	h^{- 1 / (\lambda (m - 1) + 1)} (k \zeta)
	\zeta^{1 / (\lambda (m - 1) + 1) - 1}
	\,
	d \zeta
	\ge
	C
	\left(
		\frac{
			E (r_{i+1})
		}{
			h (k E (r_{i+1}))
		}
	\right)^{1 / (\lambda (m - 1) + 1)}
$$
and $\zeta_j / \tau < r_i < \zeta_j \tau$, we obtain
\begin{align}
	&
	\int_{
		E (r_i)
	}^{
		E (r_{i + 1})
	}
	h^{- 1 / (\lambda (m - 1) + 1)} (k \zeta)
	\zeta^{1 / (\lambda (m - 1) + 1) - 1}
	\,
	d \zeta
	\ge
	C
	(r_i - r_{i+1})
	\nonumber
	\\
	&
	\qquad
	{}
	\times
	\left(
	\zeta_j^{
		- \lambda (m - 1) + n - 1
	}
		\sup_{
			(\zeta_{j + 1}, \zeta_j)
		}
		\rho 
		\frac{
			\essinf_{
				B_{\zeta_{j - 1}} \setminus B_{\zeta_{j + 2}}
			} 
			f^\lambda 
		}{
			\esssup_{
				B_{\zeta_{j - 1}} \setminus B_{\zeta_{j + 2}}
			} 
			f^{\lambda - 1}
		}
	\right)^{ 
		1 / (\lambda (m - 1) + 1)
	}
	\label{PL3.4.3}
\end{align}
for all $i \in \Xi$.
It is easy to see that
\begin{align*}
	&
	\int_{
		E (\zeta_{j - 1})
	}^{
		E (\zeta_{j + 2})
	}
	h^{- 1 / (\lambda (m - 1) + 1)} (k \zeta)
	\zeta^{1 / (\lambda (m - 1) + 1) - 1}
	\,
	d \zeta
	\\
	&
	{}
	\qquad
	{}
	\ge
	\sum_{
		i \in \Xi
	}
	\int_{
		E (r_i)
	}^{
		E (r_{i + 1})
	}
	h^{- 1 / (\lambda (m - 1) + 1)} (k \zeta)
	\zeta^{1 / (\lambda (m - 1) + 1) - 1}
	\,
	d \zeta
\end{align*}
and
$$
	\sum_{
		i \in \Xi
	}
	(r_i - r_{i + 1})
	\ge
	\zeta_j - \zeta_{j + 1}
	=
	\left(
		1
		-
		\frac{1}{\tau}
	\right)
	\zeta_j.
$$
Thus, summing~\eqref{PL3.4.3} over all $i \in \Xi$ we derive~\eqref{L3.4.2}.
\end{proof}

We also need the following known result proved in~\cite[Lemma~2.3]{meIzv}.

\begin{Lemma}\label{L3.5}
Let $\varphi : (0,\infty) \to (0,\infty)$ and $\psi : (0,\infty) \to (0,\infty)$
be measurable functions such that
$$
	\varphi (\zeta)
	\le
	\essinf_{
		(\zeta / \theta, \theta \zeta)
	}
	\psi
$$
with some real number $\theta > 1$ for almost all $\zeta \in (0, \infty)$.
Also assume that 
$0 < \mu \le 1$,
$M_1 > 0$,
$M_2 > 0$,
and
$\nu > 1$
are some real numbers with
$M_2 \ge \nu M_1$.
Then
$$
	\left(
		\int_{M_1}^{M_2}
		\varphi^{-\mu} (\zeta)
		\zeta^{\mu - 1}
		\,
		d \zeta
	\right)^{1 / \mu}
	\ge
	K
	\int_{M_1}^{M_2}
	\frac{
		d \zeta
	}{
		\psi (\zeta)
	},
$$
where the constant $K > 0$ depends only on $\mu$, $\nu$, and $\theta$.
\end{Lemma}

Is is easy to see that~\eqref{T2.1.1} implies the inequality
\begin{equation}
	\int_1^\infty
	h^{- 1 / (\lambda (m - 1) + 1)} (k \zeta)
	\zeta^{1 / (\lambda (m - 1) + 1) - 1}
	\,
	d \zeta
	< 
	\infty
	\label{PT2.1.1}
\end{equation}
for any real number $k > 0$.
To verify this, it is enough to make the change of variable $t = k \zeta$ in the left-hand side of~\eqref{T2.1.1}.
In turn,~\eqref{PT2.1.1} implies that
\begin{equation}
	\int_1^\infty
	\frac{
		d \zeta
	}{
		h (k \zeta)
	}
	<
	\infty
	\label{PT2.1.2}
\end{equation}
for any real number $k > 0$.
To show the validity of~\eqref{PT2.1.2}, it suffices to take 
$\mu = 1 / (\lambda (m - 1) + 1)$, 
$\theta = 2$, 
$\psi (\zeta) = h (k \zeta)$, 
and
$\varphi (\zeta) = h (k \zeta / 2)$
in Lemma~\ref{L3.5}.

\begin{proof}[Proof of Theorem~$\ref{T2.1}$]
Assume the converse. Let $u$ has an unremovable singularity at zero.
In this case, in view of Lemma~\ref{L3.2}, relation~\eqref{L3.3.1} holds.
Thus, by Lemma~\ref{L3.4}, there exists a positive integer $j_0$ such that
for all $j > j_0$ at list one of inequalities~\eqref{L3.4.1}, \eqref{L3.4.2} is valid.
We denote by $\Xi_1$ the set of integers $j > j_0$ for which~\eqref{L3.4.1} is valid. 
Also let $\Xi_2$ be the set of all the other integers $j > j_0$.

Since
\begin{equation}
	\zeta_j^n
	\sup_{
		(\zeta_{j + 1}, \zeta_j)
	}
	\rho 
	\frac{
		\essinf_{
			B_{\zeta_{j - 1}} \setminus B_{\zeta_{j + 2}}
		} 
		f^\lambda 
	}{
		\esssup_{
			B_{\zeta_{j - 1}} \setminus B_{\zeta_{j + 2}}
		} 
		f^{\lambda - 1}
	}
	\ge
	\int_{
		\zeta_{j + 1}
	}^{
		\zeta_j
	}
	r^{n-1}
	q (r)
	\,
	d r
	\label{PT2.1.3}
\end{equation}
for any $j > j_0$, summing~\eqref{L3.4.1} over all $j \in \Xi_1$, we have
\begin{equation}
	\int_{
		E (\zeta_{j_0})
	}^\infty
	\frac{
		d \zeta
	}{
		h (k \zeta)
	}
	\ge
	C
	\sum_{j \in \Xi_1}
	\int_{
		\zeta_{j + 1}
	}^{
		\zeta_j
	}
	r^{n - 1}
	q (r)
	\,
	d r.
	\label{PT2.1.4}
\end{equation}
At the same time, summing~\eqref{L3.4.2} over all $j \in \Xi_2$, one can conclude that
\begin{align*}
	&
	\int_{
		E (\zeta_{j_0})
	}^\infty
	h^{- 1 / (\lambda (m - 1) + 1)} (k \zeta)
	\zeta^{1 / (\lambda (m - 1) + 1) - 1}
	\,
	d \zeta
	\\
	&
	\quad
	{}
	\ge
	C
	\sum_{l \in \Xi_2}
	\left(
		\zeta_j^n
		\sup_{
			(\zeta_{l + 1}, \zeta_j)
		}
		\rho 
		\frac{
			\essinf_{
				B_{\zeta_{j - 1}} \setminus B_{\zeta_{j + 2}}
			} 
			f^\lambda 
		}{
			\esssup_{
				B_{\zeta_{j - 1}} \setminus B_{\zeta_{j + 2}}
			} 
			f^{\lambda - 1}
		}
	\right)^{ 
		1 / (\lambda (m - 1) + 1)
	},
\end{align*}
whence in accordance with~\eqref{PT2.1.3} and the inequality
\begin{align*}
	&
	\sum_{j \in \Xi_2}
	\left(
		\zeta_j^n
		\sup_{
			(\zeta_{j + 1}, \zeta_j)
		}
		\rho 
		\frac{
			\essinf_{
				B_{\zeta_{j - 1}} \setminus B_{\zeta_{j + 2}}
			} 
			f^\lambda 
		}{
			\esssup_{
				B_{\zeta_{j - 1}} \setminus B_{\zeta_{j + 2}}
			} 
			f^{\lambda - 1}
		}
	\right)^{ 
		1 / (\lambda (m - 1) + 1)
	}
	\\
	&
	\qquad
	{}
	\ge
	\left(
		\sum_{j \in \Xi_2}
		\zeta_j^n
		\sup_{
			(\zeta_{j + 1}, \zeta_j)
		}
		\rho 
		\frac{
			\essinf_{
				B_{\zeta_{j - 1}} \setminus B_{\zeta_{j + 2}}
			} 
			f^\lambda 
		}{
			\esssup_{
				B_{\zeta_{j - 1}} \setminus B_{\zeta_{j + 2}}
			} 
			f^{\lambda - 1}
		}
	\right)^{ 
		1 / (\lambda (m - 1) + 1)
	}
\end{align*}
it follows that
$$
	\left(
		\int_{
			E (\zeta_{j_0})
		}^\infty
		h^{- 1 / (\lambda (m - 1) + 1)} (k \zeta)
		\zeta^{1 / (\lambda (m - 1) + 1) - 1}
		\,
		d \zeta
	\right)^{
		\lambda (m - 1) + 1
	}
	\ge
	C
	\sum_{j \in \Xi_2}
	\int_{
		\zeta_{j + 1}
	}^{
		\zeta_j
	}
	r^{n-1}
	q (r)
	\,
	d r.
$$
Thus, summing the last estimate with~\eqref{PT2.1.4}, we obtain
\begin{align*}
	&
	\int_{
		E (\zeta_{j_0})
	}^\infty
	\frac{
		d \zeta
	}{
		h (k \zeta)
	}
	+
	\left(
		\int_{
			E (\zeta_{j_0})
		}^\infty
		h^{- 1 / (\lambda (m - 1) + 1)} (k \zeta)
		\zeta^{1 / (\lambda (m - 1) + 1) - 1}
		\,
		d \zeta
	\right)^{
		\lambda (m - 1) + 1
	}
	\\
	&
	\qquad
	{}
	\ge
	C
	\int_0^{
		\zeta_{j_0}
	}
	r^{n - 1}
	q (r)
	\,
	d r.
\end{align*}
In view of~\eqref{PT2.1.1} and~\eqref{PT2.1.2}, this contradicts~\eqref{T2.1.2}.
\end{proof}

\begin{proof}[Proof of~Theorem~$\ref{T2.2}$]
We take 
$h (t) = t^\lambda$, 
$\rho (r) = r^{\lambda (m - n)}$, 
and 
$q (r) = r^{\lambda (m - n)} z (r)$
in Theorem~\ref{T2.1}.
\end{proof}

\end{document}